\documentclass[12pt,twoside]{amsart}
\usepackage{amssymb}
\usepackage{a4wide}
\usepackage[latin1]{inputenc}
\usepackage[T1]{fontenc}
\usepackage{times}

\def\hpq0{h^{p,q}_{\leq 0}}
\def\Hpq0{\H_{\leq 0}^{p,q}}

\def\dbar{\bar\partial}
\def\ddbar{\partial\dbar}

\def\R{{\mathbb R}}

\def\C{{\mathbb C}}

\def\D{\mathcal{D}}

\def\A{{\mathcal A}}

\def\H{{\mathcal H}}

\def\Re{{\rm Re\,  }}
\def\Im{{\rm Im\,  }}

\def\be{\begin{equation}}
\def\ee{\end{equation}}

\newtheorem{thm}{Theorem}[section]

\newtheorem{prop}[thm]{Proposition}

\theoremstyle{definition}

\theoremstyle{remark}

\newtheorem{preremark}{Remark}
\newtheorem{preex}{Example}

\numberwithin{equation}{section}

\title[]
{The openness conjecture and complex Brunn-Minkowski inequalities.}

\address{Department of Mathematics\\Chalmers University
  of Technology \\
 } 

\email{ bob@chalmers.se}

\author[]{ Bo Berndtsson}

\begin{document}

\begin{abstract}We discuss recent versions of the Brunn-Minkowski inequality in the complex setting, and use it to prove the openness conjecture of Demailly and Koll\'ar. 
\end{abstract}
\maketitle
{\small \it Dedicated to Bradley Manning and Edward Snowden in recognition of their work for openness}

\section{Introduction}

Let $u$ be a plurisubharmonic function defined in the unit ball, $B$,  of $\C^n$. 
The {\it openness conjecture} of Demailly and  Koll\'ar (\cite{Demailly-Kollar}) states that the interval of numbers $p>0$ such that 
\be
\int_{rB} e^{-pu} <\infty,
\ee
for some $r>0$, is open. This is quite easy to see in one variable, since the singularities of $u$ are then given quite explicitly by the Green potential of $\Delta u$, but the higher dimensional case is much more subtle. The openness conjecture was first proved in dimension 2 by Favre and Jonsson (see \cite{Favre-Jonsson}), and then for all dimensions in \cite{0Berndtsson}. After that, simpler proofs and generalizations to the so called {\it strong openness conjecture} (see below), have been given by Guan-Zhou, \cite{Guan-Zhou} and Hiep, \cite{Hiep}, see also \cite{Lempert}. 

The proof of the openness conjecture from \cite{0Berndtsson} was based on positivity properties of certain vector bundles from \cite{2Berndtsson}, that in some ways can be seen as a complex variables generalization of the Brunn-Minkowski theorem. The aim of this survey article is to describe these 'complex Brunn-Minkowski inequalities' and explain how they can be applied in this context. In this we will basically restrict ourselves to the simplest cases and refer to the original articles for complete statements and proofs. It should be stressed that the recent proofs of Guan-Zhou and Hiep are actually simpler than the method to prove the openness conjecture presented here, but we hope that the original proof still has some interest and that the exposition here also may serve as an introduction to how our 'complex Brunn-Minkowski inequalities' can be applied. The method here also gives for free a bound on the exponents $p$, that can also be obtained in the setting of Guan-Zhou, \cite{2Guan-Zhou}, but with more work.

In the last section we also give  a very brief account of the strong openness conjecture and sketch a conjectural picture how the strong openness theorem might fit into our method.  

\section{The  Brunn-Minkowski theorem}

The classical Brunn-Minkowski theorem for convex bodies (see e g \cite{Gardner} for a nice account, including history and applications) can be formulated in the following way.
\begin{thm} Let $A_0$ and $A_1$ be two convex bodies in $\R^n$, and denote by
\be
A_t:=\{a=ta_1+ (1-t)a_0; a_0\in A_0 \, \, \text{and} \, a_1\in A_1\},
\ee
for $0\leq t\leq 1$. Let $|A|$ be the Lebesgue measure of a set $A$. Then the function
$$
t\to |A_t|^{1/n}
$$
is concave. 
\end{thm}
An equivalent  statement that relies less on the additive structure on $\R^n$, and is more suitable for the complex variants that we will describe later is the following.
\begin{thm} Let $\A$ be a convex body in $\R^{n+1}$ and denote by
\be
A_t=\{a\in \R^n; (t,a)\in \A\}.
\ee
Then the function $t\to |A_t|^{1/n}$ is concave.
\end{thm}

The equivalence of these two statements is not hard to see. In order to deduce the first statement from the second we let $\A$ be the convex hull in $\R^{n+1}$ of $\{0\}\times A_0\cup \{1\}\times A_1$, and for the converse it suffices to observe that if $A_t$ is defined by (2.2), then 
$$
 t A_1+(1-t) A_0\subset  A_t  ,
$$
if $\A$ is convex. 

There is yet another version of the theorem that will be useful for us. To prove Theorem 2.1 it actually suffices to prove the seemingly weaker inequality
$$
|A_t|\geq \min(|A_0|, |A_1|),
$$
which of course trivially follows from the concavity. This is because we can rescale the sets and use the homogenity of Lebesgue measure. Therefore we see that Theorems 2.1 and 2.2 are also equivalent to saying that 
$$
t\to \log|A_t|
$$
is concave, since this also implies the min-statement. This is sometimes called the multiplicative form of the Brunn-Minkowski theorem.

One  reason why the multiplicative form is often more useful is that it allows a functional version, known as Pr\'ekopa's theorem, \cite{Prekopa}.
\begin{thm}
Let $\phi(t,x)$ be a convex function on $\R^{n+1}=\R_t\times \R^n_x$. Then
\be
\tilde\phi(t):=\log\int_{R^n} e^{-\phi(t,x)} dx
\ee
is concave.
\end{thm}
To get the multiplicative version of Theorem 2.2 from Pr\'ekopa's theorem one let's $\phi$ be the convex function that equals 0 on $\A$ and $\infty$ on the complement of $\A$. (The reader that does not like functions that take infinite values can instead use an appropriate limit of finite functions.)
Then $\tilde\phi(t)=\log|A_t|$ so $\log |A_t|$ is concave. 

Both the Brunn-Minkowski and Pr\'ekopa theorem have a variety of proofs, often pointing in different directions of generalizations. The proof of Pr\'ekopa's theorem that is most relevant for us is the one  by Brascamp and Lieb, \cite{Brascamp-Lieb}. Their proof is based on a weighted Poincar\'e inequality, with weight function $e^{-\phi}$. This weighted Poincar\'e inequality is actually a real variable version of H\"ormander's (\cite{Hormander}) $L^2$-estimate for the $\dbar$-equation, cf \cite{Dario}. It is therefore natural to ask which, if any,  inequalities in the complex domain that correspond to Pr\'ekopa's inequality. This question was discussed in \cite{Berndtsson}, \cite{2Berndtsson} and \cite{3Berndtsson}, and in the next section we will give a very brief account of this. 
\section{ A complex variant of the Brunn-Minkowski theorem}

We shall now describe the simplest version  of the results in \cite{2Berndtsson}. Let $D$ be a pseudoconvex domain in $\C^n_z$ and let $\Omega$ be a domain in $\C_t$. If $\phi(t,z)$ is a plurisubharmonic function in $\D:=\Omega\times D$ we consider for each fixed $t$ in $\Omega$ the Bergman space
$$
A^2_t:=\{h\in H(D); \int_D |h(z)|^2 e^{-\phi(t,z)} d\lambda(z)<\infty\},
$$
equipped with the Bergman norm
$$
\|h\|_t^2=\int_D |h(z)|^2 e^{-\phi(t,z)} d\lambda(z).
$$
If we put an appropriate restriction on the growth of $\phi$, like
$$
|\phi(t,z)-\psi(z)|\leq C(t),
$$
for some fixed function $\psi$ and function $C(t)$, then all the Bergman spaces $A^2_t=A^2$ are the same as sets, but their norms vary with $t$. They therefore together make up a trivial vector bundle
$$
E:=\Omega\times A^2,
$$
with an hermitian norm $\|\cdot\|_t$. This is thus an hermitian holomorphic vector bundle, and although it has in general infinite rank it has a Chern connection and a curvature, $\Theta^E$,  as in the finite rank case, see \cite{2Berndtsson}.
\begin{thm}
The curvature of $E$ is positive (in the nonstrict sense).
\end{thm}

To understand the meaning of this statement it is not absolutely necessary to resort to the technical definition of curvature and its extension to bundles of infinite rank - although there are explicit formulas for the curvature that are sometimes of interest, see \cite{3Berndtsson}. Recall that for vector bundles of finite rank, the curvature is negative if and only if the logarithm of the norm of any holomorphic section is (pluri)subharmonic. This can be taken as the definition of negative curvature also for bundles of infinite rank, and then we say that a bundle has positive curvature if its dual has negative curvature. For the bundle $E$ above we can e g  construct holomorphic sections of its dual in the following way. Let 
$$
t\to f(t)
$$ 
be a holomorphic map from $\Omega$ to $D$. For each $t$ we then let $\xi(t)\in E_t^*$ be defined as evaluation at $f(t)$, so that 
$$
\langle \xi(t),h\rangle= h(f(t))
$$
if $h$ is in $A^2$. Clearly this defines a holomorphic section of $E^*$. The squared norm of the evaluation functional at a point $z$ for the norm on $A^2_t$ is by definition $B_t(z)$,  the (diagonal) Bergman kernel for $A^2_t$ at $z$, so 
$$
\|\xi(t)\|^2_t =B_t(f(t)).
$$
Hence Theorem 3.1 implies in particular that $\log B_t(f(t))$ is subharmonic for any holomorphic map $f$, so 
$$
\log B_t(z),
$$
 is plurisubharmonic on $\Omega\times D$ (see \cite{Berndtsson} for a more general statement). 

For line bundles it is also true that the logarithm of a holomorphic nonvanishing section is (pluri)superharmonic if the curvature is positive, but we stress that this does not hold for bundles of higher rank. Thus we have no direct statements about functions like 
$$
t\to \log\int_B |h(z)|^2  e^{-\phi(t,z)}d\lambda(z)
$$
(with $h$ holomorphic), except in special cases when the rank of the bundle is one. This actually happens in some cases. If $D=\C^n$ and $\phi(t,z)$ grows like $(n+1)\log(1+|z|^2)$ at infinity, then $A^2$ contains only constants, and we can conclude that
$$
\tilde\phi:=-\log\int_B  e^{-\phi(t,z)}d\lambda(z)
$$
is subharmonic. This is of course in close  analogy with Theorem 2.3. On the other hand, if $\phi$ does not satisfy such a bound, $\tilde\phi$ is not necessarily subharmonic, as can be see from the simple example $\phi(t,z)=|z-\bar t|^2-|t|^2=|z|^2-2\Re tz$, cf \cite{Kiselman}. Then $\tilde\phi= c_n -|t|^2$ is not subharmonic. 

Theorem 3.1 is the simplest version of what we here call 'complex Brunn-Minkowski' theorems. There are many variants of the setting and the result, the most general involving fibrations of K\"ahler manifolds, holomorphic $n$-forms and holomorphic line bundles with positively curved metrics. For our purposes here Theorem 3.1 is enough and we refer to \cite{2Berndtsson} and \cite{3Berndtsson} for proofs and generalizations. Let us just mention here, in order to make a first contact with section 2, that Theorem 3.1 can be proved by computing
$$
i\ddbar_t \|h\|^2_t
$$ 
and applying H\"ormander's theorem in a way quite similar to how Brascamp and Lieb proved Pr\'ekopa's theorem. 

On a formal level, the analogy between Theorem 3.1 and Theorem 2.3 is that we have replaced the convex function in Theorem 2.3 by a plurisubharmonic function, and the constant function 1 in Pr\'ekopa's theorem by a holomorphic function $h$. Although the statement of Theorem 3.1 has nothing to do with volumes of sets, it turns out that it can be seen as a stronger version of Theorem 2.3 and implies Theorem 2.3 as a special case. To see this we shall  apply Theorem 3.1 when $D$ and $\phi$ have some symmetry properties.

We first look at the case when $D$ and $\phi(t, \cdot)$ are invariant under the natural $S^1$-action on $\C^n$
$$
(z_1, ...z_n)\to (e^{i\theta}z_1, ...e^{i\theta}z_n):=s_\theta z.
$$
Thus we assume that $D$ is invariant under $s_\theta$ for all real $\theta$ and that for each $t$ , $\phi(t, s_\theta z)=\phi(t, z)$. First we moreover assume that $D$ is even closed under the maps $z\to \lambda z$, if $|\lambda|\leq 1$. Then $D$ can be written as 
$$
D=\{z; \psi(z)<0\},
$$
for some $\psi$,  plurisubharmonic in all of $\C^n$, that is logarithmically homogenous, i e 
$\psi(\lambda z)=\psi(z)+\log|\lambda|$ if $\lambda$ is a nonzero complex number. In particular, $D$ is Runge, so the polynomials are dense in $A^2$. 

Hence  our Bergman space $A^2$ splits as a direct sum 
$$
A^2=\bigoplus_0^\infty H_m
$$
where $H_m$ is the space of polynomials, and these spaces are orthogonal for all the norms $\|\cdot\|_t$. Therefore  we also get an orthogonal  splitting 
$$
E=\bigoplus_0^\infty E_m
$$
of the hermitian vector bundle $E$. 
Since $E$ has positive curvature it follows that all the subbundles $E_m$ are also positively curved ( see e g \cite{2Raufi} for more on this). 
In particular $E_0$ is positively curved. The fibers of $E_0$ consist of polynomials of degree zero, i e constants, so $E_0$ is a line bundle with the constant function 1 as trivializing section. Since 
$$
\|1\|_t^2 =\int_D e^{-\phi(t, z) }d\lambda(z),
$$
the positivity of $E_0$ means that 
$$
t\to \log \int_D e^{-\phi(t, z) }d\lambda(z)
$$
is superharmonic, which is a(nother)  complex version of Pr\'ekopa's theorem. It is only the positivity at the level $m=0$ that gives Pr\'ekopa-like statements, higher degrees give corresponding statements for matrices $M=(M_{\alpha, \beta})(t)$ where
$$
M_{\alpha,\beta}(t)= \int_D z^\alpha\bar z^\beta e^{-\phi(t, z) }d\lambda(z)
$$
where $|\alpha|=|\beta|=m$. More precisely, we see that
$$
\Theta_m:=i\dbar M^{-1}\partial M\geq 0
$$
as a curvature operator. 

In a similar way, the usual Pr\'ekopa theorem follows from Theorem 3.1 when we assume full toric symmetry. We only sketch this and refer to an article by Raufi, \cite{2Raufi} where this is explained and used to get a matrixvalued Pr\'ekopa theorem. We then assume that both $D$ and $\phi$ are invariant under the full torus action
$$
z\to (e^{i\theta_1}z_1,... e^{i\theta_n}z_n).
$$
Then $\phi$ depends only on $r_j=|z_j|$ and is a convex function of $\log r_j$, and  $D=T^n\times D_{\R}$ where $D_{\R}$ is logarithically convex.
The orthogonal splitting
$$
A^2=\bigoplus H_\alpha,
$$
where the sum is now over all multiindices $\alpha$, then gives Pr\'ekopa's theorem after a logarithmic change of variables. Here all multiindices $\alpha$ give the same information; the change in $\alpha$ just means that the weight function changes by a linear term. 

Even in this case of full toric symmetry, Theorem 3.1 is a bit more general than Pr\'ekopa's theorem, since we do not need to assume any symmetry in $t$:
\begin{thm} Let $\Phi(t,x_1, ...x_n)=\Phi(t, z_1, ...z_n)$ be plurisubharmonic in $\Omega\times V_\R\times{i\R^n}$ and independent of the imaginary part of $z$. Then 
$$
\tilde\Phi(t):=-\log\int_{V_\R} e^{-\Phi(t,x_1, ...x_n)} dx
$$
is subharmonic in $\Omega$. In particular, if $\Phi$ is also independent of $\Im t$, then $\tilde\Phi$ is convex.
\end{thm}

In this sense one can perhaps say that the classical Brunn-Minkowski theorem is the special case of the complex theorem when we have maximal symmetry. In another direction Theorem 3.2 can be seen as a generalization of a well known result of Kiselman, \cite{Kiselman}. 
\begin{thm}{\bf(Kiselman)} Under the same assumptions as in Theorem 3.2, let
$$
\hat\Phi(t):=\inf_{x\in V_\R}\Phi(t,x).
$$
Then $\hat\Phi$ is subharmonic.
\end{thm}
(This follows from Theorem 3.2 since $\hat\Phi=\lim_{p\to \infty} \widehat{(p\Phi)}/p$.)

\section{The openness problem}
We now return to the openness problem. We have given a plurisubharmonic function $u$ in the ball, which we assume to be negative and such that
$$
\int_B e^{-u} <\infty,
$$
where $B$ is the unit ball. For any $s>0$ we let $u_s=\max(u+s,o)=\max(u,-s)+s$. We also exend this to when $s$ is complex with $\Re s>0$ by putting $u_s=u_{\Re s}$. Then $u(s,z)=u_s(z)$ is plurisubharmonic on $\Omega\times B$, with $\Omega$ being the halfplane. 
Obviously $0\leq u_s\leq s$, so $u_s$ stays uniformly bounded for $s$ bounded. Let, for $h$ holomorphic and square integrable in the unit ball
$$
\|h\|^2_s:= \int_B |h|^2 e^{-2u_s}
$$
(note the factor 2 in the exponent!). Then $\|h\|_0$ is the standard unweighted $L^2$-norm and for  $s$ in a bounded set $\|h\|_s$ is equivalent to $\|h\|_0$. The next proposition says in particular that we can express the norm
$$
\int_B |h|^2 e^{-u}
$$
in terms of $\|h\|_s$.
\begin{prop} Assume $u < 0$ and $0<p<2$. Then for $h$ square integrable in $B$
$$
\int_B |h|^2 e^{-pu}= a_p\int_0^{\infty} e^{ps}  \|h\|^2_sds +b_p\|h\|^2_0
$$
for $a_p$ and $b_p$ suitable positive constants.
\end{prop}
\begin{proof} First note that if $x<0$
$$
\int_0^{\infty} e^{ps}e^{-2\max(x+s,0)}ds=\int_0^{-x} e^{ps}ds +\int_{-x}^{\infty} e^{-2x} e^{(p-2)s}ds= C_p e^{-px}-1/p.
$$
Applying this with $x=u$ we find that
$$
C_p\int_B |h|^2 e^{-pu}=\int_0^{\infty} e^{ps} \|h\|^2_s ds +(1/p)\int_B |h|^2,
$$
which proves the proposition. 
\end{proof}

The moral of Proposition 4.1 is that we have translated questions about the norm of a scalar valued (holomorphic) function $h$ over an $n$-dimensional domain $B$, to questions about the norm of a vector ($A^2$) valued (constant) function over the interval $(0,\infty)$. These norms $\|h\|_s$ depend only on $\Re s$ and enjoy a certain convexity property by Theorem 3.1, and we shall see how this reduces the openness problem to a problem about integrability of convex functions. The next very simple proposition illustrates the idea. 
\begin{prop}
Let $k(s)$ be a convex function on $[0,\infty)$. Then 
$$
\int _0^{\infty} e^{-k(s)} ds<\infty
$$
if and only if  
$$
\lim_{s\to\infty} k(s)/s >0.
$$
\end{prop}
\begin{proof} We may of course assume that $k(0)=0$. Then $k(s)/s$ is increasing so its limit at infinity exists. If the limit is smaller than or equal to zero, then $k(s)\leq 0$ for all $s$, so the integral cannot converge. The other direction is obvious.
\end{proof}

It follows trivially from Proposition 4.2 that if 
$$
\int_0^{\infty}e^{s-k(s)} ds<\infty,
$$
then for some $p>1$
$$
\int_0^{\infty}e^{ps-k(s)} ds<\infty.
$$
In view of Proposition 4.1 this is a version of the openness statement for onedimensional spaces.  
The next theorem is a vector valued analog of Proposition 4.2.

\begin{thm} Let $H_0$ be a (separable) Hilbert space equipped with a  family of equivalent Hilbert norms $\|\cdot\|_s$ for $\Re s\geq 0$.  Assume these norms depend only on the real part of $s$ and define a hermitian metric on the trivial bundle $\Omega\times H_0$, where $\Omega$ is the right half plane, of positive curvature.  Let $H$ be the subspace of $H_0$ of elements $h$ such that
$$
\|h\|^2:=\int_0^\infty e^s \|h\|^2_s ds<\infty.
$$
Then, for any $h$ in $H$,  $\epsilon>0$  and $s>1/\epsilon$ there is an element $h_s$ in $H_0$ such that
\be
\|h-h_s\|^2_0\leq 2\epsilon \|h\|^2,
\ee
 and
\be
\|h_s\|^2_s\leq e^{-(1+\epsilon)s}\|h\|^2_0.
\ee
\end{thm}
\begin{proof} Take $\epsilon>0$ and $s>1/\epsilon$. By assumption there is a bounded linear operator $T_s$ on $H_0$ such that
$$
\langle u,v\rangle_s=\langle T_su,v\rangle_0.
$$
By the spectral theorem (see \cite{Reed-Simon}) we can realize our Hilbert space $H_0$ as an $L^2$-space over a measure space $X$,  with respect to some positive measure $d\mu$, in such a way that 
$$
\|h\|^2_0=\int_X |h|^2 d\mu(x)
$$
and 
$$
\|h\|^2_s=\int_X |h|^2 e^{-s\lambda(x)} d\mu(x).
$$
We define $h_s$ by $h_s=\chi(x)h$, where $\chi$ is the characteristic function of the set $\lambda>(1+\epsilon)$. Let $r_s=h-h_s$. Clearly
$$
\|h_s\|^2_s=\int_{\lambda>1+\epsilon}|h|^2 e^{-s\lambda(x)}d\mu(x)\leq
e^{-(1+\epsilon)s}\int_X |h|^2d\mu= e^{-(1+\epsilon)s}\|h\|^2_0.
$$
Hence (4.2) is satisfied. For (4.1) we will  use a comparison with a flat family of metrics, which we define for $0\leq \Re t\leq s$ by
\be
|h|_t^2=\int_X |h|^2 e^{-\Re t\lambda(x)}d\mu(x).
\ee
This is  a flat metric in the sense that any element $h$ in $H_0$ can be extended holomorpically as $h_\zeta=h e^{\zeta\lambda/2}$ in such a way that
$$
\|h_\zeta\|^2_{\Re\zeta}
$$
is constant. Since $|h|_t$ coincides with $\|h\|^2_t$ for $t=0$ and $t=s$ it follows that
$$
\|h\|_t^2\geq|h|^2_t
$$
for $t$ between 0 and $s$. This is a consequence of a minimum principle for positively curved metrics that we will return to shortly. Accepting this for the moment the argument continues as follows.

Since  $r_s$ and $h_s$ are orthogonal for the scalar product defined by $|\cdot|_t$,
$$
\int_0^s e^t\|h\|^2_t dt\geq\int_0^s e^t|h|^2_t dt\geq \int_0^s e^t|r_s|^2_t dt.
$$
By the definition of $r_s$ 
$$
|r_s|^2_t\geq e^{-t(1+\epsilon)}\|r_s\|^2_0.
$$
Hence
$$
\int_0^s e^t|r_s|^2_t dt\geq \int_0^s e^{-\epsilon t}dt\|r_s\|^2_0\geq 1/(2\epsilon)\|r_s\|^2_0,
$$
since $s>1/\epsilon$. All in all
$$
\|r_s\|^2_0\leq 2\epsilon \|h\|^2
$$
so we have proved (3.4).

Let us now finally return to the minimum principle used above. For bundles of finite rank, this is a consequence of a well known theorem, see \cite{Berman-Keller}, Lemma 8.11,  and the references there. In our case, when one of the bundles is flat, the proof is actually easier, as pointed out to us by Laszlo Lempert \cite{Laszlo}, see also \cite{2Lempert}. Let $\|\cdot\|_{-t}$ and $|\cdot|_{-t}$ be the dual norms of  $\|\cdot\|_t$ and $|\cdot|_t$ respectively; they both depend only on $\Re t$. It suffices to prove that the negatively curved metric $\|\cdot\|_{-\tau}$ is smaller than the flat metric  $|\cdot|_{-\tau}$ for $\Re \tau$ between 0 and $s$. Take $0<t_0<s$. Since  $|\cdot|_{-\tau}$ is flat we can, as explained above find for any $h$ in $H_0$ a holomorphic $h_\tau$ which equals $h$ for $\tau=t_0$ and has $|h_\tau|_{-\tau}$ constant. Then 
$$
\psi(\tau)\leq \log \|h_\tau\|_{-\tau}-\log |h_\tau|_{-\tau}
$$
is subharmonic and equal to zero when $\Re \tau$ equals zero or $s$. By the maximum principle $\psi(t_0)\leq 0$ so $\|h\|_{-t_0}\leq |h|_{-t_0}$ as we wanted.
\end{proof}
We are now ready to prove the openness theorem.
\begin{thm}Let $u$ be a negative plurisubharmonic function in the unit ball $B$. Assume that
$$
\int_B e^{-u}=A<\infty.
$$
Then for $p<1+c_n/A$, where $c_n$ is a constant depending only on the dimension,
$$
\int_{B/2} e^{-pu}<\infty.
$$
\end{thm}
\begin{proof}
We apply Theorem 4.3 to $h=1$. Note first that there is a constant $\delta_n$ such that if $g$ is holomorphic in the ball and 
$$
\int_B |g|^2 \leq \delta_n
$$
then $\sup_{B/2}|g|\leq 1/2$. By Theorem 4.3, we can for any $\epsilon>0$ and $s>1/\epsilon$ find a holomorphic function $h_s$ in the ball such that
$$
\|h_s\|^2_s\leq |B|e^{-(1+\epsilon)s}
$$
and 
$$
\|h-h_s\|^2_0\leq 2\epsilon A.
$$
If $\epsilon\leq \delta_n/(2A)$ it follows that $\sup_{B/2}|h-h_s|<1/2$. Since $h$ is identically 1, this implies that $|h_s|>1/2$ on $B/2$. Hence
$$
\int_{B/2} e^{-2u_s}\leq 4\int_B |h_s|^2 e^{-2u_s}=4\|h_s\|^2_s\leq 4|B|e^{-(1+\epsilon)s}.
$$
Hence, if $p<(1+\epsilon/2)$
$$
e^{ps} \int_{B/2} e^{-2u_s}\leq Ce^{-s\epsilon/2}.
$$
Integrating from $1/\epsilon$ to infinity we find by applying Proposition 4.1 again (with $B$ replaced by $B/2$) that
$$
\int_{B/2} e^{-pu}<\infty.
$$
\end{proof}\section{A conjectural picture for strong openness}

Let us first state the 'strong openness conjecture' from Demailly, \cite{Demailly}. It says that if as before $u$ is  plurisubharmonic in the ball and  $h$ is holomorphic in the ball, then the set of $p>0$ such that

$$
|h|^2 e^{-pu}
$$
is integrable in some neighbourhood of the origin, is open. The original openness conjecture is thus the case of strong openness when $h=1$. The strong openness conjecture was proved by Guan-Zhou in \cite{Guan-Zhou} and Hiep, \cite{Hiep}. Here we will discuss how this problem might be related to the methods described above, in the simpler case when we look at integrability over a compact manifold instead of some neighbourhood of the origin. 

First of all, to motivate our discussion, let us say a few words about the openness problem in one variable. Then the subharmonic function $u$ can be written locally as the sum of a harmonic part - which does not affect the inegrability- and a potential
$$
p(z)=\int\log|z-\zeta|^2 d\mu(\zeta),
$$
where $\mu=\Delta u$ is a positive measure. It is very easy to see that 
$$
e^{-pu}
$$
is integrable in some neighbourhood of the origin if and only if $\mu(\{0\})<1$. In the same way
$$
|h|^2e^{-pu}
$$
is integrable if and only if $\mu(\{0\})<k+1$, where $k$ is the order of the zero of $h$ at the origin.

We now elaborate a little bit on the 'moral' of Proposition 4.1 as described in section 4. To simplify matters somewhat we  consider a variant of the setting in section 4, where instead of a space of holomorphic functions in the ball we look at the space of sections of a line bundle over a compact manifold, so that we are dealing with finite dimensional spaces. 

Let $L\to X$ be an ample line bundle over a compact manifold. We will consider the space $H^0(X, K_X+kL)$ of holomorphic sections of the adjoint bundles $K_X+kL$. On $L$ we consider two metrics, $\phi$ and $\phi_0$, where $\phi_0$ is positively curved and smooth whereas $\phi$ is a singular metric with $i\ddbar\phi\geq 0$. We are interested in when integrals 
$$
\int_X |h|^2 e^{-(\phi+(k-1)\phi_0)},
$$
with $h$ in $H^0(X,K_X+kL)$, are finite. Let $u=(\phi-\phi_0)$. Then 
$$
\int_X |h|^2 e^{-(\phi+(k-1)\phi_0)}=\int_X|h|^2e^{-(u+k\phi_0)},
$$
and $u$ is $\omega$-plurisubharmonic for $\omega=i\ddbar\phi_0$, i e 
$i\ddbar\phi +\omega\geq 0$.
As in the previous section we put $u_s=\max(u+\Re s,0)$ for $\Re s\geq 0$. Then $u_s$ is $\omega$-plurisubharmonic on the product of the halfplane, $\Omega$,  with $X$.   We let
$$
\|h\|^2_s:=\int_X|h|^2e^{-(2u_s+k\phi_0)}.
$$
If $k\geq 2$ then $i\ddbar_{s,X}(2u_s+k\phi_0)\geq 0$. Hence we can apply the manifold version of Theorem 3.1 from \cite{Berndtsson} and conclude that the trivial vector bundle $\Omega\times H^0(X,K_x+kL)$ is positively curved when equipped with the metric $\|\cdot\|_s$. As before we get that if $h$ is in $H^0(X,K_X+kL)$ then
\be
\int_X |h|^2 e^{-(pu+k\phi_0)}=a_p\int_0^\infty e^{ps}\|h\|^2_s ds +b_p\|h\|_0.
\ee
Since the metric on $E$ depends only on the real part of $s$ we can make a change of variables $s=-\log\zeta$, where $\zeta$ is in the punctured unit disk. Abusing language slightly we let $\|h\|_\zeta=\|h\|_s$ if $s=-\log \zeta$. Then 
\be
\int_0^\infty e^{ps}\|h\|^2_s ds= (2\pi)^{-1}\int_\Delta e^{-(p+2)\log|\zeta|}\|h\|^2_\zeta d\lambda(\zeta).
\ee
We can now extend (the trivial) bundle $E$ as a vector bundle over the entire disk, including the origin, and consider 
$$
e^{-(p+2)\log|\zeta|}\|h\|^2_\zeta
$$
as a singular metric defined over the whole disk, see \cite{Berndtsson-Paun}, \cite{Raufi}. The only serious singularities of the metric are of course at the origin. Ideally, we would now have that this singular metric has a curvature $\Theta$ which has a smooth (or almost smooth) part outside the origin, plus a singular part
$\Theta_{sing}$
supported at the origin. With respect to a holomorphic frame, $\Theta_{sing}$ would the be represented by a matrix of Dirac masses at the origin which could be diagonalized in a suitable frame. The convergence of (5.2) should then be governed by the sizes of these Dirac masses, so that 
$$
\int_X |h|^2 e^{-\phi+(k-1)\phi_0},
$$ 
is finite precisely when $h$ lies in the union of the eigenspaces corresponding to Dirac masses strictly smaller than 1. In fact, our discussion of the onedimensional openness conjecture above says precisely that this holds when the rank of the bundle is one. 

There are  obstacles to making this picture rigorous. First, Raufi \cite{Raufi}, has given an example of a positively curved singular vector bundle metric, with only an isolated singularity, whose curvature does not have measure coefficients, but contains derivatives of Dirac measures. Still, it might be true that this cannot occur if the metric is $S^1$-invariant as it is in our case. If this turns out to be so it seems likely that the rest of the argument would go through so that the study of multiplier ideals, at least over compact manifolds, could be reduced to a vector valued problem over the disk.

\def\listing#1#2#3{{\sc #1}:\ {\it #2}, \ #3.}


\begin{thebibliography}{9999}





\bibitem{Berman-Keller}\listing{Berman R, Keller, J}{Bergman geodesics}{in Complex Monge-Ampere equations and geodesics in the space of Kahler metrics, Springer Lecture Notes in Math 2038}


\bibitem{0Berndtsson}\listing{Berndtsson B}{The openness conjecture for plurisubharmonic functions }{ arXiv:1305.578}

\bibitem{Berndtsson}\listing{Berndtsson B}{Subharmonicity of the Bergman kernel and some other functions associated to pseudoconvex domains}{Ann Inst Fourier, 56 (2006) pp 1633-1662}

\bibitem{2Berndtsson}\listing{Berndtsson, B}{Curvature of vector bundles
    associated to holomorphic fibrations  }{Ann Math 169 2009, pp
    531-560 }


\bibitem{3Berndtsson}\listing{Berndtsson, B}{ Strict and nonstrict positivity of direct image bundles. }{ Math. Z. 269 (2011), no. 3-4, 1201-1218.}

\bibitem{Berndtsson-Paun}\listing{Berndtsson, B and Paun, M}{Bergman kernels and the pseudoeffectivity of relative canonical bundles.}{ Duke  Math. J. 145 ,2008 no. 2, 341-378.}


\bibitem{Brascamp-Lieb}\listing  {H J Brascamp and E H  Lieb} {On
  extensions of the Brunn-Minkowski and Pr\'ekopa-Leindler theorems,
  including inequalities for log concave functions, and with an
  application to the diffusion equation.}  {J. Functional Analysis  22
  (1976), no. 4, 366--389}


\bibitem{Dario}\listing{ Cordero-Erausquin, D} {On Berndtsson's generalization of Prékopa's theorem.}{ Math. Z. 249 (2005), no. 2, 401-410.}



\bibitem{Demailly-Kollar}\listing{Demailly, J-P and Koll\'ar, J}{Semicontinuity of complex singularity exponents and K\"ahler-Einstein metrics on Fano orbifolds}{Ann Sci Ecole Norm Sup 34 (2001) pp 525-556}

\bibitem{Demailly}\listing{Demailly, J-P}{Multiplier ideal sheaves and analytic methods in algebraic geometry}{School on vanishing theorems and effective results in algebraic geometryTrieste 2000, p 1-148}

\bibitem{Favre-Jonsson}\listing{Favre, C and Jonsson, M}{Valuations and multiplier ideals}{J Amer Math Soc 18 (2005) pp 655-684}


\bibitem{Gardner}\listing{Gardner, R J}{ The Brunn-Minkowski
    Inequality}{BAMS, 39 (2002), pp 355-405}


\bibitem{Guan-Zhou}\listing{ Guan, Q and Zhou, X}{Strong openness conjecture and related problems for plurisubharmonic functions }{ arXiv:1401.7158}

\bibitem{2Guan-Zhou}\listing{ Guan, Q and Zhou, X}{Effectiveness of Demailly's strong openness conjecture and related problems }{ arXiv:1403.7247 }

\bibitem{Hiep}\listing{Hiep, P H}{The weighted log canonical threshold}{ arXiv:1401.4833}

 \bibitem{Hormander}\listing{H\"ormander, L}{ An introduction to complex analysis in several variables} {3:d edition, North Holland 1990}

\bibitem{Jonsson-Mustata}\listing{Jonsson, M and Mustata, M}{An algebraic approach to the openness conjecture of Demailly and Koll\'ar}{ArXiv:1205/4273}

\bibitem{Kiselman}\listing{ C O Kiselman}{The partial Legendre transformation form plurisubharmonic 
functions}{ Invent. Math. 49 (1978), no. 2, 137--148}


\bibitem{Lempert-Szoke}\listing{Lempert, L and Sz\"oke, R}{Uniqueness in geometric quantization}{ArXiv:1004/4863}

\bibitem{Lempert}\listing{Lempert, L}{Modules of square integrable holomorphic germs }{ arXiv:1404.0407}

\bibitem{2Lempert}\listing{Lempert, L}{A maximum principle for hermitian (and other) metrics }{ arXiv:1309.2972}

\bibitem{Laszlo}\listing{Lempert, L}{Private communication}{}

\bibitem{Prekopa}\listing{A. Prekopa} {On logarithmic concave measures and functions}{ Acad. Sci. Math. (Szeged) 34 (1973), p. 335-343  }

\bibitem{Raufi}\listing{Raufi, H}{Singular hermitian metrics on holomorphic vector bundles}{ arXiv:1211.2948}

\bibitem{2Raufi}\listing{Raufi, H}{Log concavity for matrix-valued functions and a matrix-valued Pr\'ekopa theorem}{  arXiv:1311.7343}

\bibitem{Reed-Simon}\listing{Reed, M and Simon, B}{Methods of Mathematical Physics I, Functional Analysis}{Academic Press 1972, ISBN 0-12-585001-8}
\end{thebibliography}
\end{document}